\documentclass[12pt,reqno]{amsart}
\usepackage{amssymb,amsfonts,latexsym,amstext,amsmath,epsfig}
\usepackage[left=3.0cm,top=3.0cm,right=3.0cm,bottom=3.0cm]{geometry}

\newtheorem{theorem}{Theorem}[section]

\numberwithin{equation}{section}

\begin{document}
\title{Meyniel's conjecture on the cop number: a survey}

\author{William Baird}
\address{Department of Mathematics\\
Ryerson University\\
Toronto, ON\\
Canada, M5B 2K3} \email{w3baird@ryerson.ca}
\author{Anthony Bonato}
\address{Department of Mathematics\\
Ryerson University\\
Toronto, ON\\
Canada, M5B 2K3} \email{abonato@ryerson.ca}

\thanks{The author gratefully acknowledge support from NSERC, Mprime, and Ryerson University}

\begin{abstract}
Meyniel's conjecture is one of the deepest open problems on the cop number of a graph. It states that for a connected graph $G$ of order $n,$ $c(G) = O(\sqrt{n}).$ While largely ignored for over 20 years, the conjecture is receiving increasing attention. We survey the origins of and recent developments towards the solution of the conjecture. We present some new results on Meyniel extremal families containing graphs of order $n$ satisfying $c(G) \ge d\sqrt{n},$ where $d$ is a constant.
\end{abstract}

\maketitle

\section{Introduction}

\emph{Cops and Robbers} is a game played on a reflexive
graph; that is, vertices each have at least one loop. Multiple edges are
allowed, but make no difference to the game play, so we always assume there
is exactly one edge between adjacent vertices. There are two players
consisting of a set of \emph{cops} and a single \emph{robber}. The game is
played over a sequence of discrete time-steps or \emph{rounds},
with the cops going first in round $0$ and then playing alternate time-steps. The cops and robber occupy vertices;
for simplicity, we often identify the player with the vertex they occupy. We
refer to the set of cops as $C$ and the robber as $R.$ When a player is ready to move in a round they must
move to a neighbouring vertex. Because of the loops, players can \emph{pass}, or remain on their own vertex. Observe that any subset of $C$ may
move in a given round.

The cops win if after some finite number of rounds, one of them can occupy
the same vertex as the robber (in a reflexive graph, this is equivalent to the cop landing on the robber).
This is called a \emph{capture}. The robber
wins if he can
evade capture indefinitely. A \emph{winning strategy for the cops} is a set
of rules that if followed, result in a win for the cops. A \emph{winning
strategy for the robber} is defined analogously. Cops and Robbers is often
called a \emph{vertex-pursuit game} on graphs, for reasons that should
now be apparent to the reader.

If we place a cop at each vertex, then the cops are guaranteed to win.
Therefore, the minimum number of cops required to win in a graph $G$ is a
well-defined positive integer, named the \emph{cop number} (or \emph{%
copnumber}) of the graph $G.$ We write $c(G)$ for the cop number of a graph $%
G$. If $c(G)=k,$ then we say $G$ is $k$-\emph{cop-win}. In the special case $%
k=1,$ we say $G$ is \emph{cop-win} (or \emph{copwin}).

The game of Cops and Robbers was first considered by Quilliot \cite{q} in
his doctoral thesis, and was independently considered by Nowakowski and
Winkler \cite{nw}. The authors of \cite{nw} were told about the game by G.\ Gabor.
Both \cite{nw,q} refer only to one cop. The introduction of the
cop number came in 1984 with Aigner and Fromme \cite{af}. Many papers
have now been written on cop number since these three early works; see the
book \cite{bonato} for additional references and background on the cop
number. Cops and Robbers has even found recent application in robotics, artificial intelligence, and so-called
\emph{moving target search}; see \cite{isawa, mold}.

\subsection{Meyniel's conjecture}

Meyniel's conjecture states that if
$G\ $is a graph of order $n,$ then%
\begin{equation}
c(G)=O(\sqrt{n}).  \label{meyn}
\end{equation}%
In other words, for $n$ sufficiently large there is a constant $d>0$ such
that
\begin{equation*}
c(G)\le d\sqrt{n}.
\end{equation*}%
We will refer to (\ref{meyn}) as the \emph{Meyniel bound}. The conjecture
was mentioned in Frankl's paper \cite{frankl} as a personal communication to
him by Henri Meyniel in 1985 (see page 301 of \cite{frankl} and reference
[8]; see Figure~\ref{meypic} for a rare photograph
of Meyniel). Despite this somewhat cryptic reference, Meyniel's conjecture stands out as one of the deepest (if not
\emph{the} deepest)\ problems on the cop number. The conjecture was largely unnoticed until recently, with several new works supplying upper bounds to the cop number or solving partial cases; see \cite{bkl,ch,fkl,lu,pw,ss}.
One of the motivations of this survey is to summarize what is currently known on the problem, while supplying the requisite background for researchers
to consider its aspects (and solution!) in the future.
\begin{figure} [h]
\begin{center}
\epsfig{figure=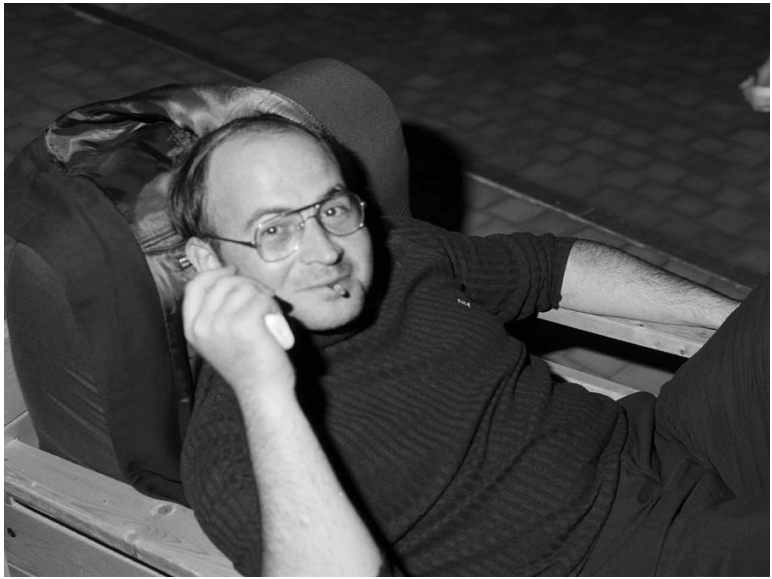, width=4in, height=3in}\caption{Henri Meyniel in Aussois, France, in the 1980's. Photo courtesy of Ge\v{n}a Hahn.}\label{meypic}
\end{center}
\end{figure}

For $n$ a positive integer, let $c(n)$ be the maximum value of $c(G)$, where
$G$ is of order $n.$ For example, $c(1)=c(2)=c(3)=1,$ while $c(4)=c(5)=2.$
Note that $c(n)$ is a non-decreasing function (to see this, note that adding a vertex of degree one does not change the cop number). We can
rephrase Meyniel's conjecture more compactly as%
\begin{equation*}
c(n)=O(\sqrt{n}).
\end{equation*}

At the heart of Meyniel's conjecture, of course, is the task of finding good upper
bounds for the cop number. Incidence graphs of projective planes show that if the conjecture is true,
then the bound is asymptotically tight (see Section~\ref{meyclose}). As a first step towards proving Meyniel's
conjecture, Frankl \cite{frankl} proved that $c(n)=o(n)$. Recent work has improved this upper bound somewhat (see Section~\ref{secchin}).
To further highlight how far we are from proving the conjecture, even the
so-called \emph{soft Meyniel's conjecture} is open, which states that for a
fixed constant $c>0,$%
\begin{equation*}
c(n)=O(n^{1-c}),
\end{equation*}

In Section~\ref{secchin} we give a history of upper bounds for the function $c(n)$. We close the section with a discussion of the conjecture in random graphs, in graph classes, and in directed graphs. We discuss families of graphs realizing the tightness of the
Meyniel bound (\ref{meyn}) in Section~\ref{meyclose}. 

For additional background and notation in graph theory, the reader is directed to the books~\cite{diestel,west}. All the graphs we consider are reflexive with no multiple edges, finite, and connected (for emphasis, we will occasionally remind the reader that the graphs under consideration are connected).

\section{Upper Bounds for $c(n)$\label{secchin}}

For many years, the best known upper
bound was the one proved by Frankl~\cite{frankl}.

\begin{theorem} \cite{frankl}
\label{ub} For $n$ a positive integer
\begin{equation*}
c(n)= O\left( n\frac{\log \log n}{\log n}\right) .
\end{equation*}
\end{theorem}

For a fixed integer $k\geq 1,$ an induced subgraph $H$ of $G$ is $k$-\emph{guardable} if, after finitely many moves, $k$ cops can move only in the
vertices of $H$ in such a way that if the robber moves into $H$ at round $t$, then he will be captured at round $t+1$. For example, a clique or a closed
neighbour set (that is, a vertex along with its neighbours) in a graph are $1$-guardable.

Given a connected graph $G,$ the \emph{distance} between vertices $u$ and $v$ in $G,$ denoted
$d_G(u,v),$ is the length of a shortest path connecting $u$ and $v.$
A path $P$ in $G$ is \emph{isometric} if for all vertices $u$ and $v$ of $P,$
\begin{equation*}
d_{P}(u,v)=d_{G}(u,v).
\end{equation*}%
For example, a shortest path (or geodesic) connecting two vertices is
isometric. The following theorem of Aigner and Fromme \cite{af} on guarding
isometric paths has found a number of applications.
\begin{theorem}\cite{af}
\label{pguard}An isometric path is $1$-guardable.
\end{theorem}

For completeness, we give a proof of Frankl's upper bound (inspired by the discussion of Lu, Peng
\cite{lu}) making use of the Moore bound, which is an important inequality
involving the order $n$ of graph, its maximum degree $\Delta ,$ and its
diameter. For simplicity, we will write $\mathrm{diam}(G)=D.$

\begin{theorem}
Let $G$ be a graph of order $n,$ with maximum degree $\Delta >2$ and
diameter $D.$ Then%
\begin{equation}
n \leq 1+\Delta \left( \frac{(\Delta -1)^{D}-1}{\Delta -2}\right) .  \label{mb}
\end{equation}
\end{theorem}

\begin{proof}[Proof of Theorem~\ref{ub}]
Each closed neighbour set of a vertex $u$ of maximum degree $\Delta $ is $1$%
-guardable. By Theorem~\ref{pguard}, an isometric path of length $D$ is also
$1$-guardable. Asymptotically, the Moore bound becomes $
n=O(\Delta ^{D}).$

By the Moore bound, both $\Delta $ and $D$ cannot be less than
$
O\left( \frac{\log n}{\log \log n}\right) .
$
In particular, there is a subset $X$ consisting of either
a closed neighbour set or isometric path of order at least
$\frac{\log n}{\log \log n}$
in $G.$ Delete $X_{1}$ to form the graph $G^{''}.$ Although graph $G''$ may be disconnected,
the robber is confined to a connected component $G'$ of this graph. The cops then move to $G'.$
Then%
\begin{equation}
c(G)\leq c(G^{\prime })+1,  \label{1guard}
\end{equation}%
since $X_{1}$ is $1$-guardable. Now proceed by induction using (\ref%
{1guard}) to derive that
\begin{eqnarray*}
c(n) &\leq &c\left( \frac{n}{2}\right) +\frac{n/2}{\frac{\log n}{\log \log n}%
} \\
&=&O\left( n\frac{\log \log n}{\log n}\right) ,
\end{eqnarray*}%
where the equality follows by a straightforward induction.
\end{proof}

The \emph{greedy approach} used above in the proof of Frankl's theorem was used by
Chinifooroshan \cite{ch} in 2008 to give an improved bound.
\begin{theorem}\cite{ch}
\label{thechin} For $n$ a positive integer
\begin{equation}
c(n)=O\left( \frac{n}{\log n}\right) .  \label{chin1}
\end{equation}
\end{theorem}
\noindent The bound (\ref{chin1}), therefore, represents the first important step
forward in proving Meyniel's conjecture in over 25 years. The key to proving
(\ref{chin1}) comes again from the notion of guarding an induced subgraph.
A \emph{minimum distance caterpillar}
(or \emph{mdc}) is an induced subgraph $H\ $of $G$ with the following
properties.
\begin{enumerate}
\item The graph $H$ is a tree.

\item There is a path $P$ in $H$ that is \emph{dominating}: that
is, for each vertex $u$ of $H$ not in $P,$ there is a vertex $v$ of $P$
joined to $u.$
\end{enumerate}
Figure~\ref{mdcfig} gives an example of a minimum distance caterpillar.
\begin{figure} [h]
\begin{center}
\epsfig{figure=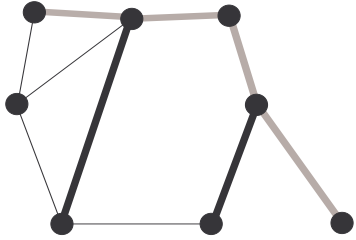, width=2in, height=1.5in}\caption{An example of an mdc,
represented by the thicker lines. The grey lines forms the path $P.$}\label{mdcfig}
\end{center}
\end{figure}

Mdc's are \textquotedblleft sticky\textquotedblright\ analogues of isometric
paths, and require just a few more cops to guard. The following theorem of \cite{ch}
can be used to prove Theorem~\ref{thechin} in a way similar to the proof of
Theorem~\ref{ub} above.
\begin{theorem}\cite{ch}
\label{mdc5}
\begin{enumerate}
\item An mdc is $5$-guardable.
\item If $G$ has order $n,$ then there is an mdc in $G$ of order at
least $\log n.$
\end{enumerate}
\end{theorem}

Let $H$ be an induced subgraph of $G.$ We say that $H$ is a \emph{retract}\index{retract} of $G$ if there is
a homomorphism $f$ from $G$ onto $H$ so that $f(x)=x$ for $x\in V(H);$ that
is, $f$ is the identity on $H.$ The map $f$ is called a \emph{retraction}.
Isometric paths are retracts in reflexive graphs: the cops stay on the image
of the robber under the retraction. If the robber moves to the subgraph,
then the cop captures the robber on his image or shadow there. One could
imagine exploiting larger retracts in graphs as an approach to proving Meyniel's conjecture. Unfortunately, this will not substantially improve
upper bounds on the cop number\index{cop number} for general graphs. A recent result from \cite%
{bkl} puts a poly-logarithmic upper bound on the order of retracts
in some graphs. The proof relies on the probabilistic method.
\begin{theorem}\cite{bkl}
For all integers $n>0,$ there is a graph of order $n$ whose largest retract
is of order $O(\log ^{8}n).$
\end{theorem}

An improvement exists to the bound (\ref{chin1}) in Theorem~\ref{thechin}.
The following theorem was proved independently by three sets of authors.
\begin{theorem}\cite{lu,fkl,ss}
\label{meylu1} For $n$ a positive integer
\begin{equation}
c(n)\le O\left( \frac{n}{2^{(1-o(1))\sqrt{\log _2n}}}\right) .  \label{meylu}
\end{equation}
\end{theorem}
\noindent The bound in (\ref{meylu}) is
currently the best upper bound for general graphs that is known, but it is
still far from proving Meyniel's conjecture or even the soft version of the
conjecture. We note that the proofs of Theorem~\ref{meylu1} in \cite{lu,ss} use the
greedy approach as in the proofs of Theorems~\ref{ub} and \ref{thechin}, while expansion properties
are used in \cite{fkl}. In addition, all of the proofs use the probabilistic method, which represents a new and interesting
approach to proving the conjecture.

\subsection{Random graphs}
As further
support for its veracity, Meyniel's conjecture has been proven
for binomial random graphs $G(n,p).$ Let $p=p(n)$ be a function of $n$ with range in $%
[0,1].$ The probability space $\mathcal{G}(n,p)=(\Omega ,\mathcal{F},\mathbb{%
P})$ of random graphs is defined so that $\Omega $ is the set of all graphs with vertex
set $[n]$, $\mathcal{F}$ is the family of all subsets of $\Omega $, and for
every $G\in \Omega $
\begin{equation*}
\mathbb{P}(G)=p^{|E(G)|}(1-p)^{{\binom{n}{2}}-|E(G)|}\,\text{.}
\end{equation*}%
The space $G(n,p)$ can be viewed as a result of ${\binom{n}{2}}$ independent
coin flips, one for each pair of vertices $\{x$,$y\}$, with the
probability that $x$ and $y$ are joined equaling $p$. We will abuse
notation and consider $G(n,p)$ as a graph, and so write $c(G(n,p))$ (note
that the cop number is a random variable on the probability space $G(n,p))$.
We say that an event holds \emph{asymptotically almost surely (a.a.s.)} if
it holds with probability tending to $1$ as $n\rightarrow \infty $.

In 2009, Bollob\'{a}s, Kun, Leader proved the following result \cite{bkl},
which proves Meyniel's bound in random graphs $G(n,p)$ up to a
multiplicative logarithmic factor for a wide range of $p=p(n).$ The basic idea behind the proof
is to surround the robber using Hall's theorem on matchings, and then use induction.
\begin{theorem} \label{bkl_t}If $p\geq 2.1\log n/n,$ then a.a.s.\
\begin{equation}
c(G(n,p))=O(\sqrt{n}\log n\,).\label{logn}
\end{equation}
\end{theorem}
Recent work by Pra{\l }at and Wormald~\cite{pw} removes the $\log n$ factor in (\ref{logn}) and hence, proves the Meyniel bound for random graphs (and also for random regular graphs).

\subsection{Graph Classes} While Meyniel's conjecture is unresolved for general graphs, we may attempt
to solve it in certain graph classes. In some cases, the extra structure
available in a class of graphs can bound the cop number from above more
easily. For example, Aigner and Fromme \cite{af} proved that $c(G)\leq 3$ if
$G$ is planar. For a fixed graph $H,$ Andreae \cite{andreae1} generalized
this result by proving that the cop number of a $K_{5}$-minor-free graph (or $K_{3,3}$%
-minor-free graph) is at most $3$ (recall that planar graphs are exactly
those which are $K_{5}$-minor-free and $K_{3,3}$-minor-free). Andreae \cite%
{andreae2} also proved that for any graph $H$ the cop number of the class of $H$-minor-free
graphs is bounded above by a constant.

Lu and Peng \cite{lu} show that the
Meyniel bound holds in the class of graphs with diameter two. The proof uses
the notion of guarding subgraph, but also uses a
randomized argument.
\begin{theorem} \cite{lu}
\label{meyd2}If $G$ is a graph on $n$ vertices with diameter two, then%
\begin{equation}
c(G)\leq 2\sqrt{n}-1.  \label{mey2}
\end{equation}
\end{theorem}
\noindent The same bound (\ref{mey2}) was also shown in \cite{lu} in the case when $G$
is bipartite and of diameter at most three.

The incidence graphs of projective planes are bipartite of diameter three,
and so show that the bound (\ref{mey2}) is asymptotically tight in that class. Meyniel extremal
families whose members have diameter two and bounded chromatic number are given in \cite{bb}.

\subsection{Directed graphs} Another direction is the analogue of Meyniel's conjecture in digraphs. For the conjecture to be sensible,
we should restrict our attention to strongly connected graphs (otherwise, a digraph can have cop number $n-1$ even
if the underlying graph is connected). Recent work by Frieze et al.\ \cite{fkl} using expansion properties shows that the cop number
of a connected digraph of order $n$ is $O(n(\log \log n)^2 /\log n).$ Can we do better? In other words, does the Meyniel bound
hold for strongly connected digraphs? For tournaments, Meyniel's bound fails to be tight. A set $D$ is dominating in a tournament, if for each vertex $x$ not in $D,$ there is a vertex $y$ in $D$ with $(y,x)$ a directed edge. The \emph{domination number} of a tournament $G$, written $\gamma (G)$, is the minimum cardinality of a dominating set. Erd\H{o}s proved (see p.\ 28 of \cite{moon}) that if $G$ is a tournament on $n$ vertices, then $\gamma (G) \le \lceil \log_2 n \rceil,$ thereby giving a logarithmic upper bound on the cop number of tournaments.

\section{Lower bounds for $c(n)$}\label{meyclose}

Meyniel's conjecture states that the cop number is at most approximately $\sqrt{n}.$
Examples are known (and will be discussed immediately below) which have cop
number very close to $\sqrt{n}.$ However, the question remains how close the
cop number can approach $\sqrt{n}$ \emph{from below}.

For graphs with large cop number, we turn to incidence graphs. An \emph{incidence structure} consists
of a set $P$ of points, and a set $L$ of lines along with an incidence relation consisting of ordered pairs of
points and lines. Given an incidence structure $S$, we define its incidence graph $G(S)$ to be the bipartite graph whose vertices consist
of the points (one color), and lines (the second color), with a point joined to a line if it is incident with it in $S.$ Incidence structures (and graphs) are quite general, but we restrict our attention to \emph{partial linear spaces}, where any pair of points (lines) is incident with at most one line (point). It is an exercise that the incidence graph of a partial linear space has diameter at least three with girth at least $6.$

Projective planes are some of the most well-studied examples of incidence structures. A \emph{projective plane} consists of a set of
points and lines satisfying the following axioms.

\begin{enumerate}
\item There is exactly one line incident with every pair of distinct points.

\item There is exactly one point incident with every pair of distinct lines.

\item There are four points such that no line is incident with more than two
of them.
\end{enumerate}

Hence, projective planes are particular partial linear spaces; condition three rules out certain degenerate cases where all points are on a single line or all lines are on a single point. We are interested in finite projective planes, which always have $q^{2}+q+1$
points for some integer $q>0$ (called the \emph{order} of the plane).

For a given projective plane $P$, define $G(P)$ to be the
bipartite graph with red vertices the points of $P,$ and the blue vertices
represent the lines. Vertices of different colors are joined if they are
incident. We call this the \emph{incidence graph of} $P.$ See Figure~\ref%
{fanogp} for $G(P), $where $P$ is the Fano plane (that is, the projective
plane of order $2)$.
\begin{figure} [h]
\begin{center}
\epsfig{figure=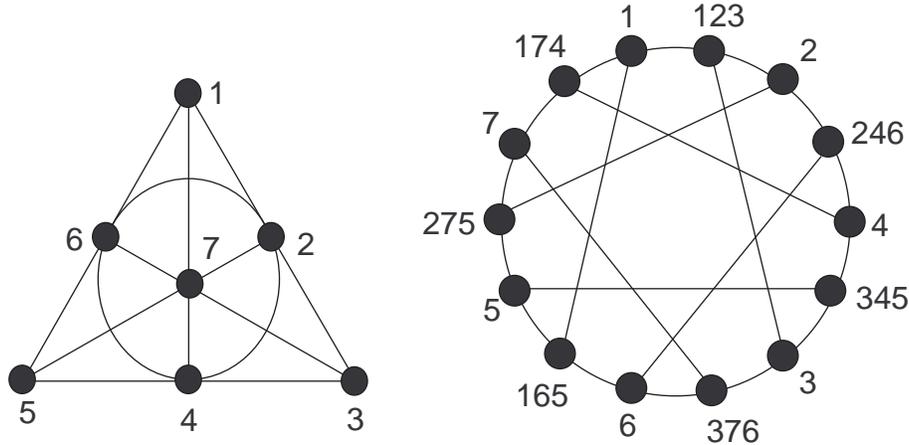}\caption{The Fano
plane and its incidence graph. Lines are represented by triples.}\label{fanogp}
\end{center}
\end{figure}
We note the incidence graph of the Fano plane is isomorphic to the famous \emph{Heawood graph}.

Aigner and Fromme proved the following theorem in \cite{af}, which provides a useful lower bound on the cop number in some graphs. The girth of a graph is the length of a shortest cycle. The \emph{minimum degree} of $G$ is written $%
\delta (G).$
\begin{theorem} \cite{af}
\label{mindeg}If $G$ has girth at least $5,$ then $c(G)\geq \delta (G).$
\end{theorem}
\noindent Hence, Theorem~\ref{mindeg} proves that $c(G(P))\ge q+1.$ As proven in \cite{pra}, we actually have that $c(G(P))=q+1$. However, the orders of $G(P)$ depend on the orders of projective planes. The only orders where projective planes are known to
exist are prime powers; indeed, this is a deep conjecture in finite geometry.
What about integers which are not prime powers? An infinite family of graphs $(G_n: n\ge 1)$ is \emph{Meyniel extremal} if there is a constant $d$ such that for all $n$, $c(G_n)\ge d\sqrt{|V(G_n)|}.$

Recall the
famous \emph{Bertrand postulate} (see \cite{che,erdosb}).
\begin{theorem}
\label{bpost}For all integers $x>1,$ there is a prime in the interval $%
(x,2x) $.
\end{theorem}
In \cite{pra}, a Meyniel extremal family was given using incidence graphs of projective planes and Theorem~\ref{bpost}. Using Bertrand's postulate, it was shown that
$$c(n)\geq \sqrt{\frac{n}{8}} $$
for $n\geq 72.$
Using this theorem and a result from number theory, it was shown in \cite{pra} that for sufficiently large $n,$%
\begin{equation}
c(n)\geq \sqrt{\frac{n}{2}}-n^{0.2625}.  \label{meyx1}
\end{equation}
We do not know if (\ref{meyx1}) is the best possible lower bound for $c(n),$ and it would be interesting to find out.

A graph is $(a,b)$-\emph{regular} if each vertex has degree either $a$ or $b.$ We provide a new construction, giving infinitely many Meyniel extremal families containing graphs which are $(a,b)$-regular for certain $a$ and $b.$ An
\emph{affine plane of order }$q$ has $q^2$-many points, each line has $q$ points, and each pair of distinct points is on a unique line. In an affine plane, there are $q^{2}+q$ lines, and each point is on $q+1$
lines. The relation of parallelism on the set of lines is an equivalence relation,
and the equivalence classes are called parallel classes. Note that each
parallel class contains $q$ lines, and there are $q + 1$ parallel classes.

\begin{theorem}
\label{mey_extr}For a prime power $q$ and all $k=o(q),$ there exist graphs of order $2q^{2}+(1-k)q$ which are $(q+1-k,q)$-regular and have cop number between $[q+1-k,q].$
\end{theorem}
\noindent Note the graphs described in Theorem~\ref{mey_extr} have order $(1-o(1))q^{2}$ with cop number $(1-o(1))q$ and so are Meyniel extremal. In particular, we can set $k=q^{1-\varepsilon },$ for $\varepsilon \in (0,1)$ and obtain infinitely many distinct Meyniel families.

\begin{proof}[Proof of Theorem~\ref{mey_extr}]
Consider an affine plane $\mathcal{A}$ with order $q.$  The incidence graph $G(\mathcal{A})$ has order $2q^{2}+q$, and
is $(q+1,q)$-regular. As $G(\mathcal{A})$ has girth at least $6,$ we have
that the graphs $\{G(\mathcal{A}):\mathcal{A}$ an affine plane of order $q\}$
form a Meyniel extremal family.

Affine planes of order $q$ may be partitioned into $(q+1)$-many parallel classes, each
containing $q$ lines. Form the partial planes $\mathcal{A}^{-k}$ by deleting the lines in
some fixed set of $k>0$ parallel classes. For a given $\mathcal{A}^{-k}$ the bipartite graph $G(\mathcal{A}%
^{-k})$ is then $(q+1-k,q)$-regular, and has order $2q^{2}+(1-k)q.$ As the
girth is at least $6,$ we have by Theorem~\ref{mindeg} that%
$$
c(G)\geq q+1-k.
$$

We claim that
\begin{equation}
c(G)\leq q.  \label{lal}
\end{equation}%
To prove (\ref{lal}), we play with $q$ cops. Fix a parallel class which was not deleted, say $\ell,$ and place one cop on each line of the parallel class. As each point is on some line in $\ell,$ the robber must move to some line $L \not \in \ell$ to avoid being captured in the first round.

Fix a point $P$ of $L,$ and let $L'$ be the line of $\ell$ which intersects $L$ at $P.$ Move the cop on $L'$ to $P.$ Now the robber cannot remain on $L$ without being captured, and so must move to some point. However, each point not on $L'$ is joined to some cop, so the robber must move to a point of $L'.$ But the unique point on $L'$ joined to $L$ is $P,$ which is occupied by a cop. \end{proof}

Recent work in \cite{bb} provides constructions of new Meyniel extremal families from designs and geometries.

\section{Acknowledgements}
The author thanks Graeme Kemkes, Richard Nowakowski, and Pawel Pra{\l }at for helpful discussions.

\end{document}